\documentclass{elsarticle}

\usepackage{amsthm}
\usepackage{amsmath}
\usepackage{amssymb}
\usepackage{txfonts}
\usepackage{tabls}
\usepackage{indentfirst}
\usepackage{pgf,tikz}
\usepackage{array}
\usepackage{enumerate}
\usepackage{cases}
\usepackage[arrow,matrix]{xy}
\usepackage{color,xcolor}
\usepackage{xypic}
 \newcolumntype{C}[1]{>{\centering\arraybackslash}p{#1}}

\newtheorem*{thmA}{Theorem 1}
\newtheorem*{thmB}{Theorem 2}
\newtheorem*{thmC}{Theorem 3}
\newtheorem{thm}{Theorem}[section]
 \newtheorem{cor}[thm]{Corollary}
 \newtheorem{lem}[thm]{Lemma}
 \newtheorem{prop}[thm]{Proposition}
 \newtheorem{ex}[thm]{Example}

 \theoremstyle{definition}
  \newtheorem{defn}[thm]{Definition}
 \theoremstyle{remark}
 \newtheorem{rem}[thm]{Remark}

\journal{Geometry and Topology}

\begin{document}

\title{Equi-variant and Stable Finite Decomposition Complexity}

\author{Jiawen Zhang\corref{notes}}
\cortext[notes]{This work was supported by the key discipline innovative talent training plan of Fudan University(EZH1411370/003/005), and the  European Research Council (ERC) grant of Goulnara ARZHANTSEVA, no. 259527.}
\ead{zhangj54@univie.ac.at}

\address{Faculty of Mathematics, University of Vienna, Oskar-Morgenstern-Platz 1, 1090 Vienna, Austria}

\begin{abstract}
  In this paper, we introduce and study various kinds of decomposition complexity. First, we give a characterization of residually finite groups having finite decomposition complexity (FDC). Secondly, we introduce equi-variant straight FDC (sFDC), and prove that a group having equi-variant sFDC if and only if its box space having sFDC. Finally, we show that elementary amenable groups have equi-variant sFDC by introducing something called stable FDC.
\end{abstract}

\begin{keyword}
Equi-variant straight FDC, Stable FDC, Residually finite groups, Elementary amenability.
\end{keyword}

\maketitle

\section{Introduction}
Finite decomposition complexity (FDC) is a concept introduced by E. Guentner, R. Tessera and G. Yu \cite{GTY12} in 2010, in order to solve certain strong rigidity problem including the stable Borel conjecture. Briefly speaking, a metric space has FDC if it admits an algorithm to decompose itself into some nice pieces which are easy to handle in certain asymptotic way. It generalizes finite asymptotic dimension, which was firstly introduced by M. Gromov in 1993 as a coarse analogue to the classical topological covering dimension, but it didn't get much attention until G. Yu proved that the Novikov higher signature conjecture holds for groups with finite asymptotic dimensions in 1998 \cite{Yu98}. FDC also implies Property A \cite{GTY12}, so by G. Yu's celebrated theorem, the coarse Baum-Connes conjecture holds for a metric space with bounded geometry and FDC \cite{yu2000coarse}. Straight finite decomposition complexity (sFDC) is a weak version of FDC introduced by A. Dranishnikov and M. Zarichnyi \cite{DZ14} in order to tell the difference between FDC and asymptotic property C.

On the other hand, given a residually finite group $G$ with a sequence of finite-index normal subgroups $\{N_k\}$, following J. Roe \cite{Roe03}, we can associate a metric space called the box space with respective to $\{N_k\}$. The idea goes back to G. A. Margulis who has firstly constructed a sequence of expanders using the box space of a group with Kazhdan's Property (T) in 1973 \cite{Mar73}. Box spaces can be viewed as a kind of coarse geometric approximation to the original group, and some relations have already been studied between a group and its box space. For example, a result of E. Guentner says that the box space has Property A if and only if the original group is amenable (see \cite{Roe03}). There are similar characterizations concerning a-T-menability \cite{chen2013characterization} and Property (T) \cite{willett2014geometric}. In this paper, we focus on FDC and straight FDC, and give some similar results. To be more precise, we prove:
\begin{thmA}\label{main1}
Let $G$ be a residually finite group with a sequence of finite-index normal subgroups $\{N_i\}$ with trivial intersection.
Then $G$ has FDC if and only if there exists a sequence of increasing positive numbers $\{r_i\}$ with $r_i \rightarrow \infty$ ($i \rightarrow \infty$), such that the coarse disjoint union of $\{B_{G/N_i}(1,r_i)\}$ has FDC. In particular, if the box space $\square_{\{N_i\}}G$ has FDC, then $G$ has FDC.
\end{thmA}

Then we introduce equi-variant sFDC, and prove:
\begin{thmB}\label{main2}
A residually finite group has equi-variant sFDC if and only if its box space has sFDC.
\end{thmB}

To give some concrete examples with equi-variant sFDC, we introduce stable FDC. And we seek groups with this property in the class of amenable ones. Amenability is a well known concept introduced by von Neumann in 1929 \cite{neumann1929allgemeinen}. He also showed that finite groups and abelian groups are amenable, and amenability is closed under the following operations: subgroup, extension, quotient and direct limit. Groups which can be obtained from finite or abelian groups after finitely many such operations are called elementary amenable (see \cite{chou1980elementary}). With our notion of stable FDC, it's rather easy to show:
\begin{thmC}\label{main3}
Elementary amenable groups have stable FDC, which implies equi-variant sFDC.
\end{thmC}

We should remark here that a recent result of M. Finn-Sell and J. Wu says more that the box space of an elementary amenable group has finite asymptotic dimension which is bounded by its Hirsch length \cite{finn2015asymptotic}, which also implies that elementary amenable groups has equi-variant sFDC. We were told by J. Wu about their result in a discussion after we have proved our theorem.

This paper is organized as follows. In Section 2, we recall some basic definitions and results in coarse geometry concerning asymptotic dimension and finite decomposition complexity; In Section 3, firstly we prove Theorem 1, then we introduce equi-variant sFDC and show some characterizations for sFDC and equi-variant sFDC. Using these characterizations, we prove Theorem 2. In the last section, we introduce stable FDC, and show some permanence properties for it, and finally we prove Theorem 3.

~\\
\noindent{\bf Acknowledgment.}
I thank Guoliang Yu, Qin Wang, Yijun Yao and Jianchao Wu for many stimulating discussions. I would particularly like to thank Romain Tessera for suggesting the concept "stable FDC" in a discussion during my visit to Paris. I would also like to thank Andrzej Zuk for his hospitality during my visit to Paris 7 University, where the last part of this paper was finished. Finally, I would like to thank my advisor, Xiaoman Chen, for continuing support and introducing this topic to me.

\section{Preliminaries}

\subsection{Asymptotic Dimension and Finite Decomposition Complexity}
In this section, we recall two conceptions in coarse geometry: asymptotic dimension and finite decomposition complexity.
Asymptotic dimension was first introduced by M. Gromov in 1993, but it didn't gather much attention until G. Yu proved that the Novikov higher signature conjecture holds for groups with finite asymptotic dimensions in 1998 \cite{Yu98}. Here we also recommend \cite{BD05} for reference. Finite decomposition complexity is a conception which generalizes finite asymptotic dimension. It was recently introduced by E. Guentner, R. Tessera and G. Yu \cite{GTY12} to solve certain strong rigidity problem including the stable Borel conjecture.

Let's begin with some notations and basic definitions.
From now on, we will use the usual letters $X, Y, Z, \cdots$ to denote metric spaces, and use letters in curlycue $\mathcal{X}, \mathcal{Y}, \mathcal{U}, \mathcal{V}, \cdots$ to denote metric families. Recall a \emph{metric family} is a family consisting of metric spaces, usually denoted by $\mathcal{X}=\{X_i\}$.

Let $X$ be a metric space and $r>0$. We call a family $\mathcal{U}=\{U_i\}$ of subsets in $X$ $r-$\emph{disjoint}, if for any $U\neq U'$ in $\mathcal{U}$, $d(U,U')\geqslant r$, where $d(U,U')=\mathrm{inf}\{d(x,x'):x\in U,x'\in U'\}$. We write
$$X=\bigsqcup_{r-disjoint}U_i$$
for this. We call a family $\mathcal{V}$ \emph{uniformly bounded}, if sup$\{\mathrm{diam}(V):V\in\mathcal{V}\}$ is finite.

\begin{defn}[\cite{gromov1992asymptotic}]
Let $X$ be a metric space. We say that the \emph{asymptotic dimension} of $X$ doesn't exceed $n$ and write asdim$X\leqslant n$, if for every $r>0$, the space $X$ can be covered by $n+1$ subspaces $X_0,X_1,\cdots,X_n$, and each $X_i$ can be further decomposed into some $r-$disjoint uniformly bounded subspaces:
$$X=\bigcup^n_{i=0}X_i,\mbox{\quad}X_i=\bigsqcup_{\mbox{\scriptsize$\begin{array}{c} r-disjoint \\ j\in \mathbb{N} \end{array}$}}X_{ij} \mbox{~~and~~}\sup_{i,j}\mbox{diam}X_{ij}<\infty.$$
We say asdim$X=n$, if asdim$X\leqslant n$ and asdim$X$ is not less than $n$.
\end{defn}

We give some examples having finite asymptotic dimensions.
\begin{ex}[\cite{NY12}, \cite{roe2005hyperbolic}]
\begin{enumerate}[1)]
  \item asdim$(\mathbb{Z}^n)$ = $n$ for all $n\in \mathbb{N}$, where $\mathbb{Z}$ is the integer number;
  \item (J. Roe) Hyperbolic group in Gromov's sense has finite asymptotic dimension.
\end{enumerate}
\end{ex}

From the definition, it's easy to see that the asymptotic dimension of a subspace is not greater than that of the whole space. There are some other equivalent definitions for asymptotic dimension, but we are not going to focus on this and guide the readers to \cite{BD05} for reference. Now we introduce the notion of finite decomposition complexity which naturally generalizes finite asymptotic dimension.

\begin{defn}[\cite{GTY12}]\label{dec def}
A metric family $\mathcal{X}$ is called \emph{$r-$decomposable} over another metric family $\mathcal{Y}$ if for any $X\in\mathcal{X}$, there exists a decomposition:
$$X=X_0 \cup X_1,\mbox{\quad}X_i=\bigsqcup_{\mbox{\scriptsize$\begin{array}{c} r-disjoint \\ j\in \mathbb{N} \end{array}$}}X_{ij},$$
where $X_{ij}\in \mathcal{Y}$. It's denoted by $\mathcal{X}\stackrel{r}{\rightarrow}\mathcal{Y}$.
\end{defn}

\begin{defn}[\cite{GTY12}]\label{FDC def1}
Define:
\begin{itemize}
  \item $\mathcal{D}_0$ is the collection of all bounded families.
  \item For any ordinal number $\alpha>0$,
        $$\mathcal{D}_\alpha=\{\mathcal{X}:\forall r>0,\exists\beta<\alpha, \exists \mathcal{Y}\in\mathcal{D}_\beta, \mbox{~such that~}\mathcal{X}\stackrel{r}{\rightarrow}\mathcal{Y}\}.$$
\end{itemize}
We call a metric family $\mathcal{X}$ has \emph{finite decomposition complexity} (FDC) if there exists some ordinal number $\alpha$ such that $X\in \mathcal{D}_\alpha$. We say that a single metric space $X$ has FDC if $\{X\}$, viewed as a metric family, has FDC. In \cite{GTY12}, it has been proved that $X$ has finite asymptotic dimension if and only if there exists a non-negative integer number $n$, such that $X\in \mathcal{D}_n$.
\end{defn}

Now we introduce another equivalent definition of FDC by the decomposition game \cite{GTY12}.
Consider the following game with two players. Given a metric family $\mathcal{X}$, roughly speaking, the aim of player 1 is to decompose $\mathcal{X}$ into some bounded family, while player 2 tries to obstruct such decompositions.

More precisely, suppose $\mathcal{Y}_0=\mathcal{X}$.
In round 1, player 1 claims that he can decompose $\mathcal{Y}_0$ arbitrarily, and player 2 challenges him by giving a positive number $R_1>0$. If player 1 can find some metric family $\mathcal{Y}_1$ and $R_1-$decomposes $\mathcal{Y}_0$ over $\mathcal{Y}_1$, then the first round is over.
In round 2, player 1 again claims that he can decompose $\mathcal{Y}_1$ arbitrarily, and player 2 challenges him by giving another positive number $R_2>0$. If player 1 can find some metric family $\mathcal{Y}_2$ and $R_2-$decomposes $\mathcal{Y}_1$ over $\mathcal{Y}_2$, then the second round is over.
Generally, in round $i$, player 1 claims that he can decompose $\mathcal{Y}_{i-1}$ arbitrarily, and player 2 challenges him by giving a positive number $R_i>0$. If player 1 can find some metric family $\mathcal{Y}_i$ and $R_i-$decomposes $\mathcal{Y}_{i-1}$ over $\mathcal{Y}_i$, then the $i-$th round is over.

We say player 1 has a \emph{winning strategy} if he can get a bounded family after finite rounds no matter what numbers player 2 gives in each round. The following diagram shows player 1 wins at round $n$:

\begin{equation*}
     \begin{array}{cccccccccc}
     Player ~1 & \mathcal{X}=\mathcal{Y}_0 & \stackrel{R_1}{\longrightarrow} & \mathcal{Y}_1 & \stackrel{R_2}{\longrightarrow} & \cdots & \stackrel{R_{n-1}}{\longrightarrow} & \mathcal{Y}_{n-1} & \stackrel{R_n}{\longrightarrow} & \mathcal{Y}_n\\
     & & & & & & & & & \\
     Player ~2 & R_1 &  & R_2 &  & \cdots &  & R_n &  & \mbox{bounded family}\\
     \end{array}
\end{equation*}\\[-4mm]

FDC can be characterized by the above decompostion game. In fact, we have
\begin{prop}[\cite{GTY12}]\label{FDC equidef}
Let $\mathcal{X}$ be a metric family, then the followings are equivalent:
\begin{itemize}
  \item $\mathcal{X}$ has FDC in the sense of Definition \ref{FDC def1};
  \item $\mathcal{X}$ admits a winning strategy.
\end{itemize}
\end{prop}

Next, we introduce some coarse permanence properties of asymptotic dimension and FDC. First we recall some basic concepts for metric families from coarse geometry. They have some well-known analogues in the case of metric space (see \cite{NY12} for example).

\begin{defn}[\cite{GTY12}]
Let $\mathcal{X},\mathcal{Y}$ be metric families.
\begin{itemize}
  \item A \emph{subspace} of $\mathcal{X}$ is a family $\mathcal{Z}$, and each element in $\mathcal{Z}$ is a subspace of some element in $\mathcal{X}$;
  \item \emph{A map of families} $\mathcal{F}$ from $\mathcal{X}$ to $\mathcal{Y}$, denoted by $\mathcal{F}: \mathcal{X} \rightarrow \mathcal{Y}$, is a collection of functions, such that each $f \in \mathcal{F}$ maps some $X$ in $\mathcal{X}$ to some $Y$ in $\mathcal{Y}$ (usually denoted by $f:X_f \rightarrow Y_f$), and each $X$ in $\mathcal{X}$ is the domain of at least one $f$ in $\mathcal{F}$;
  \item Let $\mathcal{F}: \mathcal{X} \rightarrow \mathcal{Y}$ be a map of families. The \emph{inverse image} of the subspace $\mathcal{Z}$ in $\mathcal{Y}$ is the collection $\mathcal{F}^{-1}(\mathcal{Z})=\{f^{-1}(Z): Z \in \mathcal{Z},f\in \mathcal{F}\}$;
\end{itemize}
\end{defn}

\begin{defn}[\cite{GTY12}]
Let $\mathcal{X},\mathcal{Y}$ be metric families, and $\mathcal{F}: \mathcal{X} \rightarrow \mathcal{Y}$ be a map of families.
\begin{itemize}
  \item $\mathcal{F}$ is called \emph{uniformly expansive} if there exists a non-decreasing proper function $\rho_1: \mathbb{R^+}\rightarrow \mathbb{R}$ such that for every $f\in \mathcal{F}$, and $x,x'\in X_f$,
      $$d_{Y_f}(f(x),f(x'))\leqslant \rho_1(d_{X_f}(x,x')).$$ When $\mathcal{F}=\{f\}$ consists of only one element, we also call $f$ \emph{bornologous};
  \item $\mathcal{F}$ is called \emph{effectively proper} if there exists a non-decreasing proper function $\rho_2: \mathbb{R^+}\rightarrow \mathbb{R}$ such that for every $f\in \mathcal{F}$, and $x,x'\in X_f$,
      $$\rho_2(d_{X_f}(x,x'))\leqslant d_{Y_f}(f(x),f(x')).$$ When $\mathcal{F}=\{f\}$ consists of only one element, we also call $f$ \emph{effectively proper};
  \item $\mathcal{F}$ is called a \emph{coarse embedding}, if it is both uniformly expansive and effectively proper.
  \item A coarse embedding $\mathcal{F}$ is called a \emph{coarse equivalence}, if each $f:X_f \rightarrow Y_f$ in $\mathcal{F}$ admits some $g_f: Y_f \rightarrow X_f$ and the family $\mathcal{G}=\{g_f\}$ satisfies two conditions:
      \begin{enumerate}[(1)]
        \item $\mathcal{G}$ is a coarse embedding $\mathcal{Y} \rightarrow \mathcal{X}$;
        \item there exists some constant $C>0$ such that for any $f$, and $x\in X_f$, $y\in Y_f$, we have $d(x,g_f f(x))\leqslant C$ and $d(y,f g_f(y))\leqslant C$.
      \end{enumerate}
\end{itemize}
\end{defn}

Asymptotic dimension and FDC are coarsely invariant. More precisely, we have the following two propositions.
\begin{prop}[\cite{GTY12},\cite{BD05}]\label{coarse invariance of asdim}
~\\[-0.5cm]
\begin{itemize}
  \item[(1)] Suppose two metric spaces $X$ and $Y$ are coarsely equivalent, then asdim$X$ = asdim$Y$;
  \item[(2)] Suppose two metric families $\mathcal{X}$ and $\mathcal{Y}$ are coarsely equivalent, then $\mathcal{X}$ has FDC if and only if $\mathcal{Y}$ has FDC.
\end{itemize}
 \end{prop}

We also have the following proposition for the subspace case.
\begin{prop}[\cite{GTY12},\cite{BD05}]\label{subspace for asdim}
~\\[-0.5cm]
\begin{itemize}
  \item[(1)] If $X$ is a subspace of some metric space $Y$, then $asdimX\leqslant asdimY$;
  \item[(2)] If $\mathcal{X}$ is a subspace of some metric family $\mathcal{Y}$ with FDC, then $\mathcal{X}$ also has FDC.
\end{itemize}
\end{prop}

We state another permanence property called the fibering theorem.
\begin{prop}[\cite{GTY12}]\label{fibering}
Let $\mathcal{X}, \mathcal{Y}$ be metric families and $\mathcal{F}:\mathcal{X}\rightarrow \mathcal{Y}$ be a uniformly expansive map. Assume $\mathcal{Y}$ has FDC, and for any bounded subspace $\mathcal{Z}$ in $\mathcal{Y}$, the inverse image $\mathcal{F}^{-1}(\mathcal{Z})$ has FDC. Then $\mathcal{X}$ also has FDC.
\end{prop}

Finally, we recall the notion of straight FDC (for abbreviation, sFDC). It was introduced by A. Dranishnikov and M. Zarichnyi \cite{DZ14} to compare FDC with asymptotic property C, which is a large scale version of the classical Haver's property C.
\begin{defn}[\cite{DZ14}]
A metric family $\mathcal{X}$ is said to have \emph{straight finite decomposition complexity} (sFDC), if for any increasing sequence $R_1 < R_2 < \cdots < R_n < \cdots$, there exists an integer number $m$ and a sequence of decompositions:
$$\mathcal{X}  \stackrel{R_1}{\longrightarrow}  \mathcal{Y}_1  \stackrel{R_2}{\longrightarrow}  \cdots  \stackrel{R_{m-1}}{\longrightarrow}  \mathcal{Y}_{m-1}  \stackrel{R_m}{\longrightarrow}  \mathcal{Y}_m$$
such that the family $\mathcal{Y}_m$ is bounded.
\end{defn}

Please compare the definition of FDC by decomposition game (Proposition \ref{FDC equidef}) and the one of sFDC here very carefully, to see the subtle difference. It's obvious that FDC implies sFDC. However, whether they are equivalent or not are still unknown up till now.

A. Dranishnikov and M. Zarichnyi have proved that sFDC is a coarse invariant, and they have also proved it implies Yu's Property A.
\begin{prop}[\cite{DZ14}]\label{sFDC imp A}
Let $X$ be a metric space with straight FDC. Then $X$ has Property A.
\end{prop}

\subsection{The case of groups}
Now we turned to the case of groups.

Let $G$ be a discrete group equipped with a proper length function $l$. Here "proper" means that any ball with finite radius with respect to $l$ in $G$ contains finitely many elements. Then there exists a left invariant metric $d$ on $G$ defined by $d(g,h)=l(g^{-1}h)$, where $g,h \in G$. It is a well known fact by M. Gromov \cite{gromov1992asymptotic} that for any two proper length functions $l_1,l_2$ on $G$, the identity map $id:(G,d_1) \rightarrow (G,d_2)$ is a coarsely equivalence. If $H$ is a subgroup in $G$, denoted by $H\leqslant G$, then we equip $H$ with the induced metric as a subspace of $G$.

Now let $N$ be a normal subgroup in $G$, denoted by $N \lhd G$. Define the quotient metric on $G/N$ by $d([g],[h])=d_G(gN,hN)=\min\limits_{n_1,n_2 \in N}d_G(gn_1,hn_2)$, where $[g],[h]$ are the images of $g,h \in G$ under the natural projection $\pi: G \rightarrow G/N$. Since $N$ is normal in $G$ and the metric $d_G$ is left invariant, we have:
$$d([g],[h])=d_G(gN,hN)=d_G(h^{-1}gN,N)=d_G(Nh^{-1}g,N)=d_G(h^{-1}g,N).$$

There is another equivalent way to define the quotient metric as follow. First define a length function $\bar{l}$ on $G/N$ by
$$\bar{l}([g]):=\min_{h\in G,[h]=[g]}l(h).$$
We claim the quotient metric on $G/N$ is the left invariant metric induced by $\bar{l}$: $d([g],[h])=\bar{l}([h]^{-1}[g])$.
In fact,
$$\bar{l}([h]^{-1}[g])=\bar{l}([h^{-1}g])=\min\limits_{n\in N}l(n^{-1}h^{-1}g)=\min\limits_{n\in N}d(h^{-1}g,n)=d_G(h^{-1}g,N)=d_G(gN,hN).$$

We put an easy but useful lemma here, which is important in our further proofs.

\begin{lem}\label{ball}
Let $G,N,\pi:G \rightarrow G/N$ and $l,d$ be as above. Then for any $R>0$,
$$\pi\Big(B_G(1_G,R)\Big)=B_{G/N}(1_{G/N},R).$$
Here we use the notation $B_X(x,R)$ to denote the open ball in a metric space $X$ centered at $x$ and with radius $R$.
\end{lem}

\begin{proof}
Since the quotient map is contracting, we have $\pi\Big(B_G(1_G,R)\Big)\subseteq B_{G/N}(1_{G/N},R)$.

On the other hand, for any element $[g]$ in $B_{G/N}(1_{G/N},R)$, i.e. $\bar{l}([g])< R$, by definition, $\exists n\in N$, such that $d(g,n)<R$. So $l(n^{-1}g)<R$ and $\pi(n^{-1}g)=\pi(g)$. In other words, $[g]\in \pi\Big(B_G(1_G,R)\Big)$.
\end{proof}

We have the following criterion for bornologous maps:
\begin{lem}\label{bornol}
Let $(G,l_G),(H,l_H)$ be two discrete groups equipped with proper length functions. Suppose $\phi:G\rightarrow H$ is a homomorphism, and there exists a proper function $\rho_+:[0,+\infty) \rightarrow \mathbb{R}$ satisfying: for any $R>0$,
$$\phi\Big(B_G(1_G,R)\Big) \subseteq B_H(1_H,\rho_+(R))\Big.$$
Then $\phi$ is bornologous.
\end{lem}

The proof is obvious.

\section{Equi-variant straight FDC}
In this section, we focus on residually finite groups and their box spaces. There are some well known results on the relations between large scale properties of the group and its box space. For example, a result by E. Guentner says that a group is amenable if and only if its box space has Property A (see \cite{Roe03}). X. Chen, Q. Wang and X. Wang proved a group is a-T-menable if and only if its box space can be fibred coarsely embedded into some Hilbert space \cite{chen2013characterization}. Recently, W. Rufus and G. Yu showed a group has Kazhdan's property (T) if and only if its box space has geometric property (T) \cite{willett2014geometric}. Here we want to show there are some similar relations concerning FDC and sFDC between the group and its box space.

In this section and the next one, groups are always assumed to be discrete and countable, and we will not mention this repeatedly in the following.
We begin with some basic definitions for residually finite groups and their box spaces.
\begin{defn}[\cite{Roe03}]\label{res fin}
A group $G$ is called \emph{residually finite}, if there exists a sequence of normal subgroups in $G$:
$$G=N_0 \supseteq N_1 \supseteq N_2 \supseteq \cdots,$$
such that each $N_i$ has finite index in $G$, and $\bigcap\limits_{i\geqslant 1}N_i=\{1_G\}$.
\end{defn}

Now equip $G$ with some left-invariant proper length metric, and equip every quotient group $G/N_i$ with the quotient metric (see Section 2.2). Fix a sequence of normal subgroups $\{N_i\}$ in $G$.

\begin{defn}[\cite{Roe03}]\label{box def}
The \emph{box space} of $G$ corresponding to $\{N_i\}$, denoted by $\square_{\{N_i\}}G$, is the coarse disjoint union of $\{G/N_i\}$. More explicitly, as a set, it's $\bigsqcup_i G/N_i$. The metric on each $G/N_i$ is the quotient metric, and for different pieces, define:
$$d(x,y)=diam(G/N_i)+diam(G/N_j)+i+j,$$
where $x\in G/N_i$ and $y\in G/N_j$.
\end{defn}

We state a well-known lemma which plays an important role in the proof of our main theorems. The proof is directly from the definitions of residual finiteness and quotient metric.
\begin{lem}[\cite{Roe03}]\label{basic isometry}
Let $G$ and $\{N_i\}$ be as above. Then for any $R>0$, there exists an integer number $i_0$, such that for any $i\geqslant i_0$, the restriction of the quotient map on the ball with radius $R$ and center $1_G$ is isometric, i.e. $\pi|_{B_G(1_G,R)}: B_G(1_G,R) \rightarrow B_{G/N_i}(1_{G/N_i},R)$ is an isometric bijection.
\end{lem}

From the above lemma, we can prove Theorem \ref{main1}.
\begin{proof}[Proof of Theorem 1]
If $G$ is finite, the theorem holds trivially.

We assume $G$ is infinite. Suppose $G$ has FDC, then for any increasing positive sequence $\{R_i\}$ with $R_i \rightarrow \infty$ ($i \rightarrow \infty$), the metric family $\{B_{G}(1, R_i)\}$ has FDC. Now for any positive integer number $n$, by Lemma \ref{basic isometry}, there exists an integer number $i_n$ such that the ball $B_G(1,n)$ is isometric to the ball $B_{G/N_i}(1,n)$ for any $i\geqslant i_n$ under the quotient map. We take $i_n$ to be the minimal integer satisfying the above condition, then $i_n$ goes to infinity when $n$ goes to infinity since $G$ is infinite. For any $i_n\leqslant j < i_{n+1}$, define $r_j=n$. Then the metric family $\{B_{G/N_j}(1,r_j)\}$ is isometric to a subfamily of $\{B_{G}(1, n)\}$, which has FDC. So the coarse disjoint union of $\{B_{G/N_j}(1,r_j)\}$ has FDC.

Conversely, suppose there exists a sequence of increasing positive numbers $\{r_i\}$ with $r_i \rightarrow \infty$ ($i \rightarrow \infty$), such that the coarse disjoint union of $\{B_{G/N_i}(1,r_i)\}$ has FDC. For any $R>0$, we will construct a $R-$decomposition of $G$ over some metric family having FDC, which implies that $G$ has FDC by definition. Define $A_j=\{g\in G:(4j-4)R \leqslant l(g) \leqslant (4j-2)R\}$, and $B_j=\{g\in G: (4j-2)R \leqslant l(g) \leqslant 4jR\}$ for $j=1,2,3,\ldots$. Then $\{A_j\}$, $\{B_j\}$ are $R-$disjoint, and there exists a decomposition of $G$:
$$G=(\bigsqcup_j A_j) \cup (\bigsqcup_j B_j).$$
It suffices to show the metric family $\{B_G(1,2nR):n=1,2,\ldots\}$ has FDC. Now by Lemma \ref{basic isometry}, for any integer number $n$, there exists some $i_n$, such that $r_{i_n} \geqslant 2nR$ and $\{B_G(1,r_{i_n})\}$ is isometric to $\{B_{G/N_{i_n}}(1,r_{i_n})\}$. So the metric family $\{B_G(1,2nR):n=1,2,\ldots\}$ is isometric to a subfamily of $\{B_{G/N_i}(1,r_i)\}$, which has FDC by assumption. This implies that $\{B_G(1,2nR):n=1,2,\ldots\}$ also has FDC.
\end{proof}

Theorem \ref{main1} gives a criterion for residually finite groups to tell whether they have FDC or not. Next, we would like to give another criterion for box spaces. In \cite{szabo2014rokhlin}, G. Szab\'{o}, J. Wu and J. Zacharias have shown some equivalent conditions between residually finite groups and their box spaces concerning asymptotic dimensions. To state their result as well as for our later statements, let's first introduce some notations.

\begin{defn}
Let $G$ be a group and $H$ be a subgroup in $G$.
\begin{itemize}
  \item Let $\mathcal{U}$ be a metric family consisting of subsets in $G$. $\mathcal{U}$ is called \emph{$H$-invariant}, if for any $U\in \mathcal{U}$ and any $h\in H$, $U\cdot h$ still belongs to $\mathcal{U}$;
  \item Let $\mathcal{X}, \mathcal{Y}$ be metric families consisting of subsets in $G$. For $R>0$, we call a decomposition $\mathcal{X}\stackrel{R}{\rightarrow}\mathcal{Y}$ \emph{$H$-invariant} if for any $X\in\mathcal{X}$, there exists a decomposition:
      $$X=X_0 \cup X_1,\mbox{\quad}X_i=\bigsqcup_{\mbox{\scriptsize$\begin{array}{c} R-disjoint \\ j\in \mathbb{N} \end{array}$}}X_{ij},$$
      where $X_{ij}\in \mathcal{Y}$ and $\{X_{0j}\}$, $\{X_{1j}\}$ are both $H$-invariant.
\end{itemize}
\end{defn}

We recall the result by G. Szab\'{o}, J. Wu and J. Zacharias.
\begin{prop}[\cite{szabo2014rokhlin}]
Let $G$ be a residually finite group with a sequence of normal subgroups $\{N_i\}$ satisfying the conditions in Definition \ref{res fin}. Then the following statements are equivalent:
\begin{enumerate}
    \item The box space $\square_{\{N_i\}}G$ has asymptotic dimension at most $s$;
    \item For any $R>0$, there exists an integer number $K$ and a covering of $G$: $\mathcal{U}=\mathcal{U}^0 \cup \cdots \cup \mathcal{U}^s$ with Lebesgue number at least $R$ such that each $\mathcal{U}^i$ has mutually disjoint members and is $N_K$-invariant.
\end{enumerate}
\end{prop}

Now we turn to the case of straight FDC. The idea is similar, but one needs some more technical analyses. We state Theorem 2 again explicitly.

\begin{thmB}
Let $G$ be a residually finite group with a sequence of normal subgroups $\{N_i\}$ satisfying the conditions in Definition \ref{res fin}. Then the following statements are equivalent:
\begin{enumerate}
    \item[(1)] The box space $\square_{\{N_i\}}G$ has sFDC;
    \item[(2)] For any increasing sequence $R_1 < R_2 < \cdots < R_n < \cdots$, there exist integer numbers $m$ and $K$, and a sequence of decompositions:
          $$ G  \stackrel{R_1}{\longrightarrow}  \mathcal{Y}_1  \stackrel{R_2}{\longrightarrow}  \cdots  \stackrel{R_m}{\longrightarrow}  \mathcal{Y}_m$$
          such that $\mathcal{Y}_m$ is bounded and the last decomposition $\mathcal{Y}_{m-1}  \stackrel{R_m}{\longrightarrow}  \mathcal{Y}_m$ is $N_K$-invariant. If $G$ satisfies the above condition, we call $G$ has \emph{equi-variant sFDC}.
\end{enumerate}
\end{thmB}

To prove Theorem 2, we need a new kind of decomposition.
\begin{defn}
Let $\mathcal{X}, \mathcal{Y}$ be metric families. We call $\mathcal{X}$ is $R-$\emph{full decomposable} over $\mathcal{Y}$, denoted by $\mathcal{X}\stackrel{R}{\rightsquigarrow}\mathcal{Y}$,
if for any $X\in \mathcal{X}$, there exists a decomposition:
$$X=X_0 \cup X_1,\mbox{\quad}X_i=\bigsqcup_j X_{ij},$$
where $X_{ij}\in \mathcal{Y}$ and $\{X_{ij}\}$, viewed as a cover of $X$, has Legesbue number $L(\{X_{ij}\})\geqslant R$. (Sometimes we call the decomposition in Definition \ref{dec def} "ordinary" to tell it from full decomposition.)

Furthermore, suppose $\mathcal{X}, \mathcal{Y}$ consist of subsets of some group $G$, and $H$ is a subgroup in $G$. If $\{X_{0j}\}$, $\{X_{1j}\}$ are both $H$-invariant, then we call that the full decomposition $\mathcal{X}\stackrel{R}{\rightsquigarrow}\mathcal{Y}$ is $H$\emph{-invariant}.
\end{defn}

We will show the definition of sFDC given by the original sequence of decompositions is equivalent to the one given by the sequence of full decompositions. To be more precise, we prove the following lemmas.

\begin{lem}\label{dec lem1}
Given a metric space $X$ and a sequence of decompositions:
$$ X \stackrel{R_1}{\longrightarrow}  \mathcal{Y}_1  \stackrel{R_2}{\longrightarrow}  \cdots  \stackrel{R_n}{\longrightarrow}  \mathcal{Y}_n,$$
where $0 < R_1 < R_2 < \cdots < R_n$, and $\mathcal{Y}_1, \cdots, \mathcal{Y}_n$ are metric families consisting of subsets in $X$.
Then there is a sequence of full decompositions:
$$X \stackrel{R_1/4}{\rightsquigarrow} \widetilde{\mathcal{Y}_1} \stackrel{R_2/4}{\rightsquigarrow} \widetilde{\mathcal{Y}_2}  \stackrel{R_3/4}{\rightsquigarrow}  \cdots  \stackrel{R_n/4}{\rightsquigarrow}  \widetilde{\mathcal{Y}_n},$$
where $\widetilde{\mathcal{Y}_1}, \cdots, \widetilde{\mathcal{Y}_n}$ are some metric families consisting of subsets in $X$.
\end{lem}

Before proving the above lemma, let's fix some notations. Let $(X, d)$ be a metric space, $Y \subseteq X$ be a subspace. Given $R \geqslant 0$, define the $R-$neighborhood of $Y$ in $X$ to be
$$N_{R,X}(Y)=\{x\in X: d_X(x,Y)< R\}.$$
Given $R<0$, define the $R-$neighborhood of $Y$ in $X$ to be
$$N_{R,X}(Y)=\{x \in X: d_X(x, X \setminus Y) \geqslant R\}.$$

\begin{proof}[Proof of Lemma \ref{dec lem1}]
Since $X \stackrel{R_1}{\longrightarrow}  \mathcal{Y}_1$, by definition, we have:
$$X=\Big(\bigsqcup_{R_1-disjoint}Y_{0j}\Big) \cup \Big(\bigsqcup_{R_1-disjoint}Y_{1j}\Big),$$
where $Y_{ij}\in \mathcal{Y}_1$, which implies
$$X=\Big(\bigsqcup N_{R_1/2,X}(Y_{0j})\Big) \cup \Big(\bigsqcup N_{R_1/2,X}(Y_{1j})\Big).$$
And the cover $\{N_{R_1/2,X}(Y_{ij}):i=0,1; j\in \mathbb{N}\}$ has Lebesgue number $\geqslant R_1/4$.
Define $\widetilde{\mathcal{Y}_1}$ to be the cover $\{N_{R_1/2,X}(Y):Y \in \mathcal{Y}_1\}$, then by definition, we have
$X \stackrel{R_1/4}{\rightsquigarrow} \widetilde{\mathcal{Y}_1}$.

Now let's turn to the second step of the decomposition $\mathcal{Y}_1 \stackrel{R_2}{\longrightarrow}  \mathcal{Y}_2$. In other words, for any $Y_1 \in \mathcal{Y}_1$, we have the decomposition:
$$Y_1=\Big(\bigsqcup_{R_2-disjoint}Z_{0j}\Big) \cup \Big(\bigsqcup_{R_2-disjoint}Z_{1j}\Big)$$
for some $Z_{ij}\in \mathcal{Y}_2$. It's easy to see:
\begin{eqnarray*}
N_{R_1/2,X}(Y_1) &=& \Big(\bigsqcup_{R_2-R_1 disjoint}N_{R_1/2,X}(Z_{0j})\Big) \cup \Big(\bigsqcup_{R_2-R_1 disjoint}N_{R_1/2,X}(Z_{1j})\Big)\\
                 &=& \Big(\bigsqcup_{R_2-R_1 disjoint}N_{R_1/2,N_{R_1/2,X}(Y_1)}(Z_{0j})\Big) \cup \Big(\bigsqcup_{R_2-R_1 disjoint}N_{R_1/2,N_{R_1/2,X}(Y_1)}(Z_{1j})\Big),
\end{eqnarray*}
where $N_{R_1/2,N_{R_1/2,X}(Y_1)}(Z_{ij})$ is the $R_1/2-$neighborhood of $Z_{ij}$ in $N_{R_1/2,X}(Y_1)$. More precisely, $N_{R_1/2,N_{R_1/2,X}(Y_1)}(Z_{ij})=\{x\in N_{R_1/2,X}(Y_1): d_X(x,Z_{ij})< R_1/2\}$.

Now we can do the same thing as in the first step, and get a decomposition:
\begin{eqnarray*}
N_{R_1/2,X}(Y_1) &=&\Big(\bigsqcup N_{(R_2-R_1)/2,N_{R_1/2,X}(Y_1)} \big(N_{R_1/2,N_{R_1/2,X}(Y_1)}(Z_{0j})\big)\Big) \\
                 & &\cup~~ \Big(\bigsqcup N_{(R_2-R_1)/2,N_{R_1/2,X}(Y_1)} \big(N_{R_1/2,N_{R_1/2,X}(Y_1)}(Z_{1j})\big)\Big).
\end{eqnarray*}
Note that $N_{(R_2-R_1)/2,N_{R_1/2,X}(Y_1)} \big(N_{R_1/2,N_{R_1/2,X}(Y_1)}(Z_{ij})\big)  \subseteq  N_{R_2/2,N_{R_1/2,X}(Y_1)}(Z_{ij})$, and since $\{Z_{ij}: j \in \mathbb{N}\}$ is $R_2$-disjoint for $i=0,1$, we get the following decomposition:
$$N_{R_1/2,X}(Y_1) = \Big(\bigsqcup N_{R_2/2,N_{R_1/2,X}(Y_1)}(Z_{0j})\Big) \cup \Big(\bigsqcup N_{R_2/2,N_{R_1/2,X}(Y_1)}(Z_{1j})\Big).$$
Define $\widetilde{\mathcal{Y}_2}$ to be
$$\{N_{R_2/2,\widetilde{Y}}(Z) : Z \in \mathcal{Y}_2, Z \subseteq \widetilde{Y} \mbox{~for~some~}\widetilde{Y} \in \widetilde{\mathcal{Y}_1}\}.$$
Then $\widetilde{\mathcal{Y}_1} \stackrel{R_2/4}{\rightsquigarrow}  \widetilde{\mathcal{Y}_2}$.

Inductively, we can define a sequence of families by $\widetilde{\mathcal{Y}_k} = \{N_{R_k/2,\widetilde{Y}}(Z):Z \in \mathcal{Y}_k, Z \subseteq \widetilde{Y} \mbox{~for~some~}\widetilde{Y} \in \widetilde{\mathcal{Y}_{k-1}}\}$. And we have a sequence of full decompositions:
$$X \stackrel{R_1/4}{\rightsquigarrow} \widetilde{\mathcal{Y}_1} \stackrel{R_2/4}{\rightsquigarrow} \widetilde{\mathcal{Y}_2}  \stackrel{R_3/4}{\rightsquigarrow}  \cdots  \stackrel{R_n/4}{\rightsquigarrow}  \widetilde{\mathcal{Y}_n}.$$
\end{proof}

Conversely, we have the following lemma.
\begin{lem}
Given a metric space $X$ and a sequence of full decompositions:
$$X \stackrel{R_1}{\rightsquigarrow} \widetilde{\mathcal{Y}_1} \stackrel{R_2}{\rightsquigarrow}  \cdots  \stackrel{R_n}{\rightsquigarrow}  \widetilde{\mathcal{Y}_n},$$
where $0 < R_1 < R_2 < \cdots < R_n$ and $\widetilde{\mathcal{Y}_1}, \cdots, \widetilde{\mathcal{Y}_n}$ are some metric families consisting of subsets in $X$.
Then there is a sequence of decompositions:
$$ X \stackrel{R_1}{\longrightarrow}  \mathcal{Y}_1  \stackrel{R_2-R_1}{\longrightarrow}  \cdots  \stackrel{R_n-R_{n-1}}{\longrightarrow}  \mathcal{Y}_n,$$
where $\mathcal{Y}_1, \cdots, \mathcal{Y}_n$ are metric families consisting of subsets in $X$.
\end{lem}

\begin{proof}
Since $X \stackrel{R_1}{\rightsquigarrow} \widetilde{\mathcal{Y}_1}$, by definition, we have
$$X=\Big(\bigsqcup Y_{0j}\Big) \cup \Big(\bigsqcup Y_{1j}\Big),$$
where $Y_{ij} \in \widetilde{\mathcal{Y}_1}$, and the cover $\{Y_{ij}\}$ of $X$ has Lebesgue number $\geqslant R_1$, which implies
$$X=\Big(\bigsqcup_{R_1-disjoint}N_{-R_1,X}(Y_{0j})\Big) \cup \Big(\bigsqcup_{R_1-disjoint}N_{-R_1,X}(Y_{1j})\Big).$$
Define $\mathcal{Y}_1$ to be $\{N_{-R_1,X}(Y):Y \in \widetilde{\mathcal{Y}_1}\}$, then we have $X \stackrel{R_1}{\longrightarrow}  \mathcal{Y}_1$.

Now let's turn to the second step of the decomposition: $\widetilde{\mathcal{Y}_1} \stackrel{R_2}{\rightsquigarrow} \widetilde{\mathcal{Y}_2}$. In other words, for any $Y_1 \in \widetilde{\mathcal{Y}_1}$, we have the decomposition:
$$Y_1=\Big(\bigsqcup Z_{0j}\Big) \cup \Big(\bigsqcup Z_{1j}\Big)$$
for some $Z_{ij}\in \widetilde{\mathcal{Y}_2}$ and the cover $\{Z_{ij}\}$ of $Y_1$ has Lebesgue number $L(\{Z_{ij}\}) \geqslant R_2 > R_1$.

Now we claim: $\{N_{-R_1,X}(Z_{ij})\}$ is a cover of $N_{-R_1,X}(Y_1)$ with Lebesgue number not less than $R_2 - R_1$. In particular, we have the decomposition:
$$N_{-R_1,X}(Y_1)=\Big(\bigsqcup N_{-R_1,X}(Z_{0j})\Big) \cup \Big(\bigsqcup N_{-R_1,X}(Z_{1j})\Big).$$
In fact, since $L(\{Z_{ij}\}) \geqslant R_2$, for any $y \in N_{-R_1,X}(Y_1)$, there exists some $Z_{ij}$ such that $B(y,R_2) \cap Y_1 \subseteq Z_{ij}$. We want to show:
$$B(y,R_2-R_1) \cap N_{-R_1,X}(Y_1) \subseteq N_{-R_1,X}(Z_{ij}).$$
For any $z \in B(y,R_2-R_1) \cap N_{-R_1,X}(Y_1)$, we have $d(z,y) \leqslant R_2-R_1$, and $d(z, X \setminus Y_1) \geqslant R_1$. For any $w \in X \setminus (B(y,R_2) \cap Y_1)$, if $d(w,z)<R_1$, then $w \in Y_1$, and $d(w,y) \leqslant d(w,z)+d(z,y)<R_2$, which implies $w\in B(y,R_2) \cap Y_1$. This is a contradiction, so $d(z, X \setminus (B(y,R_2) \cap Y_1)) \geqslant R_1$. Thus from the choice of $Z_{ij}$, we obtain:
$$d(z,X \setminus Z_{ij}) \geqslant d(z, X \setminus (B(y,R_2) \cap Y_1) \geqslant R_1,$$
which implies $z\in N_{-R_1,X}(Z_{ij})$, so the claim holds.

Now use the same method as in the first step, we have:
$$N_{-R_1,X}(Y_1)=\Big(\bigsqcup N_{R_1-R_2,N_{-R_1,X}(Y_1)}\big( N_{-R_1,X}(Z_{0j})\big)\Big) \cup \Big(\bigsqcup N_{R_1-R_2,N_{-R_1,X}(Y_1)}\big( N_{-R_1,X}(Z_{1j})\big)\Big),$$
and the family $\{N_{R_1-R_2,N_{-R_1,X}(Y_1)}\big( N_{-R_1,X}(Z_{0j})\big)\}$ and $\{N_{R_1-R_2,N_{-R_1,X}(Y_1)}\big( N_{-R_1,X}(Z_{1j})\big)\}$ are \\
$(R_2-R_1)-$disjoint. Define:
$$\mathcal{Y}_2 = \{N_{R_1-R_2,N_{-R_1,X}(Y)}( N_{-R_1,X}(Z)): Z \in \widetilde{\mathcal{Y}_2}, Y \in \widetilde{\mathcal{Y}_1}, \mbox{and~}Z \subseteq Y\},$$
and we have $\mathcal{Y}_1  \stackrel{R_2-R_1}{\longrightarrow}  \mathcal{Y}_2$.

Inductively, we get a sequence of decompositions:
$$ X \stackrel{R_1}{\longrightarrow}  \mathcal{Y}_1  \stackrel{R_2-R_1}{\longrightarrow}  \cdots  \stackrel{R_n-R_{n-1}}{\longrightarrow}  \mathcal{Y}_n,$$
so the lemma holds.
\end{proof}

\begin{rem}
In the above two lemmas, when $X$ is a group, and if the original sequence of ordinary (or full) decompositions is $H-$invariant, then the obtained sequence of full (or ordinary) decompositions is still $H-$invariant, where $H$ is some subgroup.
\end{rem}

Combine the above two lemmas and the remark, we have the following characterizations of sFDC and equi-variant sFDC.
\begin{prop}\label{equ1}
Let $X$ be a metric space. Then $X$ has sFDC if and only if for any increasing sequence $R_1 < R_2 < \cdots < R_n < \cdots$, with $\lim_{n \rightarrow \infty}R_n = \infty$, there exists an integer number $m$ and a sequence of full decompositions:
$$X \stackrel{R_1}{\rightsquigarrow} \mathcal{Y}_1 \stackrel{R_2}{\rightsquigarrow}  \cdots  \stackrel{R_m}{\rightsquigarrow}  \mathcal{Y}_m,$$
such that the family $\mathcal{Y}_m$ is bounded.
\end{prop}

\begin{prop}\label{equ2}
Let $G$ be a residually finite group with a sequence of normal subgroups $\{N_i\}$ satisfying the conditions in \ref{res fin}. Then $G$ has equi-variant sFDC if and only if
for any increasing sequence $R_1 < R_2 < \cdots < R_n < \cdots$, with $\lim_{n \rightarrow \infty}R_n = \infty$, there exist integer numbers $m$ and $K$, and a sequence of full decompositions:
$$G \stackrel{R_1}{\rightsquigarrow} \mathcal{Y}_1 \stackrel{R_2}{\rightsquigarrow}  \cdots  \stackrel{R_m}{\rightsquigarrow}  \mathcal{Y}_m,$$
such that $\mathcal{Y}_m$ is bounded and the last full decomposition $\mathcal{Y}_{m-1}  \stackrel{R_m}{\rightsquigarrow}  \mathcal{Y}_m$ is $N_K$-invariant.
\end{prop}

Now we have the tools in hand to prove Theorem 2.

\begin{proof}[Proof of Theorem 2]
~\\
$(1) \Rightarrow (2)$:
By the assumption and Proposition \ref{equ1}, given an increasing sequence:
$$R_1 < R_2 < \cdots < R_n < \cdots,$$
with $\lim\limits_{n \rightarrow \infty}R_n = \infty$, there exists an integer number $m$ and a sequence of full decompositions:
$$\square_{\{N_i\}}G \stackrel{R_1}{\rightsquigarrow} \mathcal{Y}_1 \stackrel{R_2}{\rightsquigarrow}  \cdots  \stackrel{R_m}{\rightsquigarrow}  \mathcal{Y}_m,$$
such that the family $\mathcal{Y}_m$ is bounded.

Since $\mathcal{Y}_m$ is bounded, we can find a constant $c$ with $c>\sup\{diam(Y): Y\in \mathcal{Y}_m\}$. We can also assume $c>R_m$.
Denote the natural projection map by $\pi_j: G \rightarrow G/N_j$.
Pick $K$ large enough such that $N_K \cap B_G(1_G,2c)=\{1_G\}$, which implies that for any $Y \subseteq G$ with $diam(Y) < c$, $\pi_K|_Y: Y \rightarrow \pi_K(Y)$ is an isometric bijection.  Also assume $d(\bigsqcup\limits_{j=1}^{K-1} G/N_j, G/N_K)\geqslant c$.

By the above conditions, for any $Y \in \mathcal{Y}_m$, if $Y \cap (G/N_K) \neq \emptyset$, then $Y \subseteq G/N_K$, and its pullback in $G$ has the form:
$$\pi_K^{-1}(Y)=\bigsqcup\limits_{h\in N_K} U_Y\cdot h,$$
where $U_Y$ is a subset in $G$ such that the restriction of $\pi_K$ on $U_Y$ is isometric.
Define the metric family $\widetilde{\mathcal{Y}_m}$ to be $\{U_Y \cdot h: Y\in \mathcal{Y}_m \mbox{~and~} Y \subseteq G/N_K, h\in N_K\}$, and $\widetilde{\mathcal{Y}_i}$ to be $\{\pi_K^{-1}(Y' \cap (G/N_K)):Y'\in \mathcal{Y}_i \mbox{~and~} Y' \cap (G/N_K) \neq \emptyset\}$ for $1\leqslant i\leqslant m-1$. We claim there exists a sequence of full decompositions as follows:
$$ G \stackrel{R_1}{\rightsquigarrow}  \widetilde{\mathcal{Y}_1}  \stackrel{R_2}{\rightsquigarrow}  \cdots  \stackrel{R_m}{\rightsquigarrow}  \widetilde{\mathcal{Y}_m},$$
and $\widetilde{\mathcal{Y}_{m-1}} \stackrel{R_m}{\rightsquigarrow}  \widetilde{\mathcal{Y}_m}$ is $G_K$-invariant.

In fact, for any $Y' \in \mathcal{Y}_{m-1}$ such that $Y' \cap (G/N_K) \neq \emptyset$, by assumption, we have a decomposition:
$$Y' \cap (G/N_K)=\Big(\bigsqcup_j Y'_{0j}\Big) \cup \Big(\bigsqcup_j Y'_{1j}\Big),$$
for some $Y'_{ij} \in \mathcal{Y}_m$ and $Y'_{ij} \subseteq G/N_K$. So
$$\pi_K^{-1}(Y' \cap (G/N_K)) = \Big(\bigsqcup_j \pi_K^{-1}(Y'_{0j})\Big) \cup \Big(\bigsqcup_j \pi_K^{-1}(Y'_{1j})\Big).$$
By the above analysis, $\pi_K^{-1}(Y'_{ij}) = \bigsqcup\limits_{h\in N_K} U_{Y'_{ij}}\cdot h$ for some $U_{Y'_{ij}} \subseteq G$, so we have:
$$\pi_K^{-1}(Y' \cap (G/N_K)) = \Big(\bigsqcup_{\mbox{\scriptsize$\begin{array}{c} h \in N_K \\ j\in \mathbb{N} \end{array}$}}U_{Y'_{0j}}\cdot h \Big) \cup \Big(\bigsqcup_{\mbox{\scriptsize$\begin{array}{c} h \in N_K \\ j\in \mathbb{N} \end{array}$}} U_{Y'_{1j}}\cdot h \Big).$$
And the cover $\{U_{Y'_{ij}}\cdot h: i=0,1; j \in \mathbb{N}; h\in N_K\}$ of $\pi_K^{-1}(Y' \cap (G/N_K))$ has Lebesgue number not less than $R_m$.
This implies $\widetilde{\mathcal{Y}_{m-1}} \stackrel{R_m}{\rightsquigarrow}  \widetilde{\mathcal{Y}_m}$ is $G_K$-invariant.

Now for any $Y'' \in \mathcal{Y}_{m-2}$ such that $Y'' \cap (G/N_K) \neq \emptyset$, by assumption, we have a decomposition:
$$Y'' \cap (G/N_K)=\Big(\bigsqcup_j Y''_{0j}\cap (G/N_K)\Big) \cup \Big(\bigsqcup_j Y''_{1j}\cap (G/N_K)\Big),$$
for some $Y''_{ij} \in \mathcal{Y}_{m-1}$, which implies
$$\pi_K^{-1}(Y'' \cap (G/N_K)) = \Big(\bigsqcup_j \pi_K^{-1}(Y''_{0j} \cap (G/N_K))\Big) \cup \Big(\bigsqcup_j \pi_K^{-1}(Y''_{1j} \cap (G/N_K))\Big).$$
And the cover $\{\pi_K^{-1}(Y''_{ij} \cap (G/N_K)): i=0,1; j \in \mathbb{N}\}$ of $\pi_K^{-1}(Y'' \cap (G/N_K))$ has Lebesgue number not less than $R_{m-1}$.
Inductively, the claim holds. Now by Proposition \ref{equ2}, $G$ has equi-variant sFDC.

$(2) \Rightarrow (1)$:
By assumption and Proposition \ref{equ2}, given an increasing sequence:
$$R_1 < R_2 < \cdots < R_n < \cdots$$
with $\lim\limits_{n \rightarrow \infty}R_n = \infty$, there exist integer numbers $m$, $K$, and a sequence of full decompositions:
$$G \stackrel{R_1}{\rightsquigarrow} \mathcal{Z}_1 \stackrel{R_2}{\rightsquigarrow}  \cdots  \stackrel{R_m}{\rightsquigarrow}  \mathcal{Z}_m,$$
such that the family $\mathcal{Z}_m$ is bounded and the last full decomposition $\mathcal{Z}_{m-1}  \stackrel{R_m}{\rightsquigarrow}  \mathcal{Z}_m$ is $N_K$-invariant.

As in the above step, choose a constant $c$ with $c>\sup\{diam(Z): Z\in \mathcal{Z}_m\}$ and $c>R_m$. Pick $K' \geqslant K$ sufficiently large such that $N_{K'} \cap B_G(1_G,2c)=\{1_G\}$, and $d(\bigsqcup\limits_{j=1}^{K'-1} G/N_j, G/N_{K'})\geqslant c$. Denote the natural projection map by $\pi_j: G \rightarrow G/N_j$.

Since $\mathcal{Z}_{m-1} \stackrel{R_m}{\rightsquigarrow}  \mathcal{Z}_m$ is $N_{K'}$-invariant, we can assume for any $Z\in \mathcal{Z}_{m-1}$, there exists a decomposition of the form:
$$Z=\Big(\bigsqcup_{\mbox{\scriptsize$\begin{array}{c} h\in N_{K'} \\ j\in \mathbb{N} \end{array}$}}U^{(K')}_{Z,j}\cdot h\Big) \cup \Big(\bigsqcup_{\mbox{\scriptsize$\begin{array}{c} h\in N_{K'} \\ j\in \mathbb{N} \end{array}$}}V^{(K')}_{Z,j}\cdot h\Big),$$
for some $U^{(K')}_{Z,j},V^{(K')}_{Z,j}\in \mathcal{Z}_m$, which implies for any $h\in N_{K'}$, $Z \cdot h = Z$. Obviously, we have the decomposition:
$$\pi_{K'}(Z)=\Big(\bigsqcup_j \pi_{K'}(U^{(K')}_{Z,j})\Big) \cup \Big(\bigsqcup_j \pi_{K'}(V^{(K')}_{Z,j})\Big).$$
And by the choice of $c$ and $K'$, the cover $\{\pi_{K'}(U^{(K')}_{Z,j}),\pi_{K'}(V^{(K')}_{Z,j}):j \in \mathbb{N}~\}$ of $\pi_{K'}(Z)$ has Lebesgue number not less than $R_m$. Similarly, for any $k\geqslant K'$, we have:
$$Z=\Big(\bigsqcup_{\mbox{\scriptsize$\begin{array}{c} h\in N_k \\ j\in \mathbb{N} \end{array}$}}U^{(k)}_{Z,j}\cdot h\Big) \cup \Big(\bigsqcup_{\mbox{\scriptsize$\begin{array}{c} h\in N_k \\ j\in \mathbb{N} \end{array}$}}V^{(k)}_{Z,j}\cdot h\Big),$$
for some $U^{(k)}_{Z,j},V^{(k)}_{Z,j}\in \mathcal{Z}_m$, which implies a decomposition:
\begin{equation}\label{dec equ1}
\pi_k(Z)=\Big(\bigsqcup_j \pi_k(U^{(k)}_{Z,j})\Big) \cup \Big(\bigsqcup_j \pi_k(V^{(k)}_{Z,j})\Big),
\end{equation}
and the cover $\{\pi_k(U^{(k)}_{Z,j}),\pi_k(V^{(k)}_{Z,j}):j~\}$ of $\pi_k(Z)$ has Lebesgue number not less than $R_m$.

Now define $\widetilde{\mathcal{Z}_i} = \bigcup\limits_{k\geqslant K'}\pi_k(\mathcal{Z}_i) \cup \{\bigsqcup\limits_{j=1}^{K'-1} G/N_j\}$ for $1\leqslant i \leqslant m$.
By Equation (\ref{dec equ1}), we have:
$$\widetilde{\mathcal{Z}_{m-1}} \stackrel{R_m}{\rightsquigarrow}  \widetilde{\mathcal{Z}_m}.$$
Now for any $Z' \in \mathcal{Z}_{m-2}$, by assumption, we have a decomposition:
$$Z'=\Big(\bigsqcup_j U'_{Z,j}\Big) \cup \Big(\bigsqcup_j V'_{Z,j}\Big),$$
for some $U'_{Z,j}, V'_{Z,j} \in \mathcal{Z}_{m-1}$. So for any $h\in N_{K'}$, we have:
$$U'_{Z,j} \cdot h = U'_{Z,j},~~V'_{Z,j} \cdot h = V'_{Z,j},$$
which implies there exists a decomposition:
$$\pi_k(Z')=\Big(\bigsqcup_j \pi_k(U'_{Z,j})\Big) \cup \Big(\bigsqcup_j \pi_k(V'_{Z,j})\Big)$$
for any $k \geqslant K'$, and the Lebesgue number of the cover $\{\pi_k(U'_{Z,j}), \pi_k(V'_{Z,j}):j \in \mathbb{N}~\}$ is not less than $R_{m-1}$.
Inductively, we get the following sequence of full decompositions:
$$\square_{\{N_i\}}G  \stackrel{R_1}{\rightsquigarrow}  \widetilde{\mathcal{Z}_1}  \stackrel{R_2}{\rightsquigarrow}  \cdots  \stackrel{R_m}{\rightsquigarrow}  \widetilde{\mathcal{Z}_m}.$$
By Proposition \ref{equ1}, $\square_{\{N_i\}}G$ has sFDC.
\end{proof}

\begin{cor}
If a residually finite group has equi-variant sFDC, then it's amenable.
\end{cor}

\begin{proof}
By Theorem 2, we know its the box space has sFDC, which implies the box space has Property A by Proposition \ref{sFDC imp A}. Now by the result of E. Guentner (\cite{Roe03}, which we have also mentioned at the beginning of this section), we obtain that the group is amenable.
\end{proof}

Now a natural question arises: does there exists any "non-trivial" residually finite group having equi-variant sFDC? Furthermore, one can ask: does there exist any "non-trivial" residually finite group whose box space has FDC? To answer these questions, we will introduce a new concept of decomposition complexity in the next section.

\section{Stable FDC}
At the end of the above section, we put two questions. In this section, we will answer them by introducing a new concept called stable FDC. This property is preserved under extension and direct union, and it's actually the motivation we introduce it.

\begin{defn}
A group $G$ is called to have \emph{stable FDC}, if the metric family $\{H/K: K \lhd H \leqslant G\}$ has FDC in the normal sense. Here we use $H \leqslant G$ and $K \lhd H$ to represent that $H$ is a subgroup in $G$ and $K$ is a normal subgroup in $H$, respectively. We call $\{H/K: K \lhd H \leqslant G\}$ \emph{the family associated with $G$}.
\end{defn}

To avoid ambiguities, we explain the metric on the family $\{H/K: K \lhd H \leqslant G\}$ in detail. First, equip $G$ with any proper length function, and define the metric on $H$ to be the induced metric, and the one on $H/K$ to be the quotient metric. For analysis on these metric, see Section 2.2. To make the definition of stable FDC proper, we need to show it's independent of the proper length function we choose on $G$.

\begin{lem}\label{coarse equi}
Let $G$ be a group, and $l_1,l_2$ be two proper length functions on $G$. Then the metric family $\{H/K:K \lhd H \leqslant G\}$ induced by $l_1$ and the one induced by $l_2$ are coarsely equivalent. Consequently, stable FDC is independent of the length function on $G$.
\end{lem}

\begin{proof}
In Section 2.2, we have already known that $(G,l_1)$ and $(G,l_2)$ are coarsely equivalent. Furthermore, it's easy to construct two proper functions $f,g:[0,+\infty) \rightarrow \mathbb{R}$ such that for any $R>0$, we have:
$$B_{(G,l_2)}(1_G,g(R)) \subseteq B_{(G,l_1)}(1_G,R) \subseteq B_{(G,l_2)}(1_G,f(R)).$$
For any subgroup $H$ in $G$, we intersect every item in the above inequality with $H$:
$$B_{(H,l_2)}(1_H,g(R)) \subseteq B_{(H,l_1)}(1_H,R) \subseteq B_{(H,l_2)}(1_H,f(R)).$$
Then for any normal subgroup $K$ in $H$, by Lemma \ref{ball}, we have:
$$B_{(H/K,\bar{l}_2)}(1_{H/K},g(R)) \subseteq B_{(H/K,\bar{l}_1)}(1_{H/K},R) \subseteq B_{(H/K,\bar{l}_2)}(1_{H/K},f(R)).$$
From Lemma \ref{bornol} and  the above inequality, the lemma holds.
\end{proof}

From the definition, it's obvious that stable FDC is preserved by taking subgroups:
\begin{lem}
Let $G$ be a group with stable FDC. For any subgroup $H$ in $G$, $H$ also has stable FDC.
\end{lem}

Stable FDC is also preserved by taking quotient:
\begin{lem}\label{quotient}
Let $G$ be a group and $N_0$ be a normal subgroup in $G$. Suppose $G$ has stable FDC, then $G/N_0$ also has stable FDC.
\end{lem}

\begin{proof}
Any subgroup in $G/N_0$ is of the form $H/N_0$ for some subgroup $H$ in $G$ containing $N_0$, and any normal subgroup in $H/N_0$ is of the form $K/N_0$ for some normal subgroup $K$ in $H$ containing $N_0$. We have the natural isomorphism:
$$\psi_{H,K}:(H/N_0)\big/(K/N_0) \cong H/K.$$
So we get a map of families $\psi:\{(H/N_0)\big/(K/N_0): K \lhd H \leqslant G\} \rightarrow \{H/K:K \lhd H \leqslant G\}$.
By Lemma \ref{ball}, $\psi_{H,K}$ maps the ball $B_{(H/N_0)\big/(K/N_0)}(1,R)$ to the ball $B_{H/K}(1,R)$ for any $R>0$.
Using Lemma \ref{bornol} twice with respect to $\psi$ and $\psi^{-1}$, we see that $\psi$ is coarsely equivalent. Since $\{H/K:K \lhd H \leqslant G\}$ has FDC, $\{(H/N_0)\big/(K/N_0):K \lhd H \leqslant G\}$ also has FDC. In other words, $G/{N_0}$ has stable FDC.
\end{proof}

Now we prove stable FDC is preserved under extension and direct union.
\begin{prop}\label{extension}
Let $1 \rightarrow N \rightarrow G \rightarrow Q \rightarrow 1$ be an extension, and $N,Q$ have stable FDC. Then $G$ also has stable FDC.
\end{prop}

\begin{proof}
Without losing generality, assume $N$ is a normal subgroup in $G$ and $Q=G/N$. Let $\{H/K: K \lhd H \leqslant G\}$ be the family associated with $G$. We want to show it has FDC.
Define a map of families $\phi$ from the family associated with $G$ to the one associated with $G/N$: for any $K \lhd H \leqslant G$, define
$$\phi_{H,K}:H/K \rightarrow (HN/N)\big/(KN/N)$$
to be the composition of
$$H/K \rightarrow HN/KN \rightarrow (HN/N)\big/(KN/N), hK \mapsto hKN \mapsto \bar{h}(KN/N),$$
where $\bar{h}$ is the image of $h$ under the projection $HN \rightarrow HN/N$.
It's easy to see $\phi_{H,K}$ is contracting, so the map of families $\phi$ is uniformly expansive. By fibering theorem \ref{fibering}, we only need to show the family $\{\mathrm{Ker}\phi_{H,K}\}$ has FDC.

By calculation, $\mathrm{Ker}\phi_{H,K}=H \cap KN/K=K(H \cap N)/K$. We define another map of families $\varphi$: for any $K \lhd H \leqslant G$, define
$$\varphi_{H,K}: K(H \cap N)/K \rightarrow H \cap N / K\cap H \cap N$$
to be the natural isomorphism.

Equip $G$ with a proper length function, and equip the normal subgroup $N$ with the induced length function. By Lemma \ref{ball}, it's easy to see $\varphi_{H,K}$ maps the ball $B_{K(H \cap N)/K}(1,R)$ to the ball $B_{H \cap N / K\cap H \cap N}(1,R)$ for any $R>0$.
Since $\varphi_{H,K}$ is an isomorphism, by Lemma \ref{bornol}, $\varphi$ is a coarse equivalence. Finally by the assumption on $N$, we see the family $\{\mathrm{Ker}\phi_{H,K}\}$ has FDC.
\end{proof}

We turn to the case of direct union.
\begin{prop}\label{direct union}
Let $\{G_n\}$ be a sequence of groups with $G_n \subseteq G_{n+1}$ for any $n$, and $G=\bigcup\limits_n G_n$ be the direct union of $\{G_n\}$. Assume each $G_n$ has stable FDC, then $G$ has stable FDC.
\end{prop}

\begin{proof}
Given any $R>0$, since the length function $l$ on $G$ is proper, there exists some integer number $m>0$ such that the ball $B_G(1,R)$ is contained in $G_m$. For any $K \lhd H \leqslant G$, we also have $B_{H/K}(1,R) \subseteq (G_m \cap H)K/K$ by Lemma \ref{ball}. Consider the coset decomposition:
$$H/K=\bigsqcup_{\lambda \in \Lambda}h_{\lambda}(G_m \cap H)K/K,$$
where $\Lambda$ is the set of representatives of the cosets, and different cosets have distance greater than or equal to $R$ by the choice of $m$.
In other words, the family associated with $G$ can be $R$-decomposed over the family $\{h_{\lambda}(G_m \cap H)K/K: \lambda \in \Lambda, K \lhd H \leqslant G\}$, which is coarsely equivalent to $\{(G_m \cap H)K/K: K \lhd H \leqslant G\}$. By Lemma \ref{ball} and the similar argument in the proof of the above proposition, $\{(G_m \cap H)K/K: K \lhd H \leqslant G\}$ is coarsely equivalent to $\{G_m \cap H \big/ G_m \cap H \cap K: K \lhd H \leqslant G\}$, which is a subfamily of $\{H'/K':K' \lhd H' \leqslant G_m\}$. Since $G_m$ has stable FDC by assumption, we see the family $\{G_m \cap H \big/ G_m \cap H \cap K: K \lhd H \leqslant G\}$ has FDC. So $G$ has stable FDC.
\end{proof}

The concept stable FDC we have just introduced seems to be rather strong. One may ask: does there exist some "nontrivial" group with this property? Now we will show that all elementary amenable groups have stable FDC. We begin with the most simple case.
\begin{prop}\label{fin and abel}
Finite groups and Abelian groups have stable FDC.
\end{prop}

\begin{proof}
Finite groups naturally have stable FDC since the metric we choose on the quotient group is the quotient metric. Now turn to Abelian case. By Proposition \ref{direct union}, we can assume $G$ is finitely generated. So by the structure theorem of finitely generated Abelian group \cite{lang2002algebra} and Lemma \ref{coarse equi}, we can assume $G=\mathbb{Z}^n$. Now by Proposition \ref{extension}, it's sufficient to prove $\mathbb{Z}$ has stable FDC. In other words, the family $\{\mathbb{Z},\mathbb{Z}/m\mathbb{Z}:m\in \mathbb{N}\}$ has FDC. This is obvious since for any $R>0$, the family can be $R-$decomposed over a bounded family using canonical decompositions.
\end{proof}

Now we can prove the first part of Theorem 3. In other words:
\begin{prop}\label{elementary}
Elementary amenable groups have stable FDC.
\end{prop}

\begin{proof}
Let $EG$ be the smallest class of groups that contains finite groups and Abelian groups, and is closed under extension and direct union.
By a theorem of C. Chou \cite{chou1980elementary}, a group is elementary amenable if and only if it's in $EG$. Now from Proposition \ref{extension}, \ref{direct union}, and \ref{fin and abel}, we see the theorem holds.
\end{proof}

Finally, we return to the case of residually finite groups to answer the questions raised at the end of the above section. Let $G$ be a residually finite group with a sequence of normal subgroups $\{N_i\}$ satisfying the conditions in Definition \ref{res fin}. The following lemma is obvious by definition of FDC and the metric defined on the box space (see Definition \ref{box def}).
\begin{lem}\label{box lem}
Let $G$ and $\{N_i\}$ be as above. Then the box space $\square_{\{N_i\}}G$ has FDC if and only if the metric family $\{G/N_i\}$ has FDC. In particular, if $G$ has stable FDC, then the box space $\square_{\{N_i\}}G$ has FDC.
\end{lem}

Combine Proposition \ref{elementary} with Lemma \ref{box lem}, we obtain Theorem 3.

From Theorem 1 and Theorem 3, we have reproved that elementary amenability implies FDC, which was originally proved in \cite{GTY12}.

To sum up the main results, we give a diagram concerning some relations between the properties introduced in this paper. Here $G$ is a residually finite group, and $\square G$ is its box space corresponding to some sequence of normal subgroups.

~\\
\resizebox{!}{0.25cm}
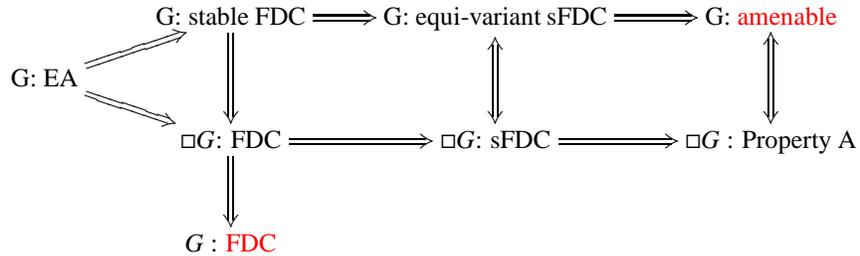
{
  \xymatrix@R=0.35cm{
     ~                 &  \mbox{G: stable FDC}  \ar@{=>}[dd]_{}   \ar@{=>}[r]_{}  &  \mbox{G: equi-variant sFDC}  \ar@{=>}[r]_{} \ar@{<=>}[dd]_{}  &  \mbox{G: \textcolor[rgb]{1.00,0.00,0.00}{amenable}} \ar@{<=>}[dd]_{} \\
  \mbox{G: EA}  \ar@{=>}[ur]_{} \ar@{=>}[dr]_{}    &      ~       &                     ~              &      ~     \\
     ~                 &  \square G \mbox{: FDC} \ar@{=>}[dd]_{} \ar@{=>}[r]_{}     &  \square G \mbox{: sFDC}      \ar@{=>}[r]_{}   &  \square G \mbox{~: Property A}  \\
     ~&~&~&~\\
     ~                 &  G \mbox{~: \textcolor[rgb]{1.00,0.00,0.00}{FDC}}          &   ~   &  ~  }
}

~\\
\noindent {\bf Problem:}
It is well known that FDC does not imply amenability (consider the free group for example). However, does amenability imply FDC? (See the red items in the above diagram.)

\section*{References}
\bibliographystyle{plain}
\bibliography{bibfileFDC}

\end{document}